\documentclass{amsart}
\usepackage{amsmath,amsfonts,euscript,amssymb,amsthm,graphicx,subfigure,verbatim, esint}
\usepackage{mathrsfs}
\usepackage[active]{srcltx}
\usepackage[colorlinks,linkcolor={blue},
citecolor={blue},urlcolor={blue},]{hyperref}
\usepackage{hyperref}

\newtheorem{theorem}{Theorem}
\newtheorem{corollary}{Corollary}
\newtheorem{lemma}{Lemma}
\newtheorem{proposition}{Proposition}

\makeatletter
\renewcommand{\subsubsection}{\@startsection{subsubsection}{3}{\z@}%
  {0.8ex \@plus 0.5ex \@minus 0.2ex}%
  {0.5ex \@plus 0.2ex}%
  {\normalfont\bfseries}}
\makeatother

\newcommand{\semiGroup}[1]{S(#1)}

\newcommand{\absOd}[1]{\abs{#1}_{od}}

\newcommand{\FG}{d_{\frac{1}{2}}}

\newcommand{\FLhalf}{(-\Delta)^{\frac{1}{2}}}

\newcommand{\FLquarter}{(-\Delta)^{\frac{1}{4}}}

\newcommand{\FL}[1]{(-\Delta)^{#1}}

\renewcommand{\d}{\, \mathrm{d}}

\newcommand{\dx}{\d x}
\newcommand{\dy}{\d y}
\newcommand{\dz}{\d z}
\newcommand{\dt}{\d t}
\newcommand{\ds}{\d s}

\renewcommand{\dh}{\d h}

\newcommand{\R}{\mathbb{R}}

\newcommand{\N}{\mathbb{N}}

\newcommand{\Sm}{\mathbb{S}^{m-1}}

\newcommand{\scp}[2]{\left\langle #1, #2 \right\rangle}

\newcommand{\eps}{\varepsilon}
\newcommand{\del}{\partial}

\newcommand{\norm}[2]{\left\| #1 \right\|_{ #2 }}
\newcommand{\tnorm}[2]{\left\lvert\!\left\lvert\!\left\lvert #1 \right\rvert\!\right\rvert\!\right\rvert_{#2}}
\newcommand{\snorm}[2]{\left[ #1 \right]_{#2}}
\DeclareMathOperator*{\esssup}{ess\,sup}

\newcommand{\abs}[1]{\left\vert #1 \right\vert}

\begin{document}

\title[Well-posedness of half-harmonic map heat flows]{Well-posedness of half-harmonic map heat flows for rough initial data}

\author{Kilian Koch}
\thanks{Kilian Koch (corresponding author): Lehrstuhl für Angewandte Analysis, RWTH Aachen University, Kreuzherrenstr. 2, 52062 Aachen, Germany. Email: \texttt{koch@math1.rwth-aachen.de}. Phone: +49 241 80-94586.}

\author{Christof Melcher}
\thanks{Christof Melcher: Lehrstuhl für Angewandte Analysis, RWTH Aachen University, Kreuzherrenstr. 2, 52062 Aachen, Germany. Email: \texttt{melcher@rwth-aachen.de}. Phone: +49 241 80-94585.}

\date{\today}

\begin{abstract}
   We adopt the Koch-Tataru theory for the Navier-Stokes equations, based on Carleson measure estimates, to develop a scaling-critical low-regularity framework for half-harmonic map heat flows. This nonlocal variant of the harmonic map heat flow has been studied recently in connection with free boundary minimal surfaces. We introduce a new class of 
   initial data for the flow, broader than the conventional energy or Sobolev spaces considered in previous work, for which we establish existence, uniqueness, and continuous dependence. The class particularly includes homogeneous initial data that give rise to self-similar expanders.
\end{abstract}

\maketitle

\section{Introduction}
Fractional harmonic maps and flows arise at the intersection of several areas of mathematics, appearing as continuum limits of certain spin systems and in connection with the Plateau problem for disc-type minimal surfaces. Fractional harmonic maps generalize classical harmonic maps by replacing the governing Dirichlet energy with a fractional version 
introducing non-local interactions between points at various scales, see e.g. \cite{DaLio, DaLioRiviere, MillotSire, Schikorra2012}. Notably, fractional harmonic flows are governed by evolution equations that model the time evolution of maps under non-local diffusion processes. Due to their deep and unexpectedly intricate mathematical structure, fractional harmonic flows have recently attracted significant attention in the geometric analysis community, see e.g. \cite{Lenzmann2018, Struwe2024, Wettstein2021, Wettstein2022, Wettstein2023, Wright2024}. In this work, we consider the initial value problem for half-harmonic map heat flows from euclidean spaces into spheres $u:\R^n \times (0,T) \to \Sm$, governed by the equation
\[
\del_t u + (-\Delta)^{\frac 1 2} u \perp T_u \Sm
\]
and initial conditions $a:\R^n \to \Sm$ in suitable function spaces. Here $(-\Delta)^{\frac 1 2}$ denotes the usual fractional Laplacian defined, e.g., via Fourier transform.
The equation corresponds to the $L^2$ gradient flow of the square of the fractional Sobolev norm
\begin{equation} \label{eq:energy}
\mathcal{E}_{\frac 1 2} (u) = \frac{\gamma_n}{4} \int_{\R^n} \int_{\R^n} \frac{|u(x) - u(y)|^2}{|x-y|^{n+1}} \, dx dy. 
\end{equation} 
where $\gamma_n := \Gamma((n+1)/2)/\pi^{(n+1)/2}
$. The abstract weak form of the flow equation reads
\[
\scp{\del_t u}{\phi}_{L^2} + \mathcal{E}_{\frac 1 2}'(u)\langle \phi \rangle =0 
\] 
for all smooth tangent fields $\phi \perp u$. Fundamental questions of the existence and uniqueness of half-harmonic flows from $\mathbb{S}^1$ have been examined in \cite{Wettstein2021, Wettstein2022, Wettstein2023, Struwe2024, Wright2024}, including local well-posedness in the optimal Sobolev space $H^{\frac 1 2}$ and a bubbling analysis at possible blow-up points that allows for an extension to a global weak solution in full analogy to the well-known result for the conventional harmonic map heat flow from surfaces \cite{Struwe1985}. Infinite time blow-up has indeed been established in \cite{SireWeiZheng} with a conjecture of non-occurance of finite time singularities
for half-harmonic map heat flows as purely non-local phenomenon. Existence and partial regularity results for global weak solutions in higher dimensions were obtained in \cite{Schikorra2017, HyderSireWang}. 
\\\\
In this work, we shall examine well-posedness in more general class
scale invariant spaces larger than the energy space $\dot{H}^{\frac 1 2}(\R; \Sm)$, or more generally the critical Sobolev $\dot{H}^{\frac n 2}(\R^n; \Sm)$ as examined in \cite{MelcherSakellaris, KK2024}. One motivation arises from situations which require rougher initial conditions, such as homogeneous fields that involve point or jump discontinuities at the origin. These are crucial for the construction of self-similar expanders, as explored in \cite{GermainRupflin, DeLaire} for the conventional harmonic map heat flow and the Landau-Lifshitz-Gilbert equation, respectively.
Another motivation of more general character is to explore the conditions under which solutions exist or may fail to exist, as well as under what conditions solutions remain smooth or develop singularities, a well-known program in the context of the Navier-Stokes equation. A somewhat optimal framework established in \cite{KochTataru} is to exploit the connection between heat semigroups and $\mathrm{BMO}$ spaces through their maximal characterization in terms of Carleson measures \cite{Stein1993, GrafakosLoukas2014, LemariePierre2002}. The concept was also realized for a wide range of geometric flows \cite{KochLamm} and other nonlinear parabolic evolution equation, see e.g. \cite{BloemkerRomito, DeLaire, Seis, LiMelcher}, 
including in particular harmonic map heat flows and their coupling to fluid dynamic equations in models of liquid crystals \cite{Wang2010, Wang2012}. For the harmonic map heat flow
$\del_t u -\Delta u= |\nabla u|^2 u$ admissible initial conditions
$a: \R^n \to \R^m$ 
are characterized via their caloric extension $u: \R^{n+1}_+ \to \R^m$, i.e. $u(t)=g_t \ast a$ where $g_t$ is the heat kernel, so that, for some time horizon $T>0$, local energy averages
\[
\sup_{0<r<\sqrt{T}}\sup_{x \in \R^n} \int_0^{r^2} \fint_{B_r(x)} |\nabla u|^2
\]
satisfy a certain smallness condition. This is nothing but a local Carleson condition on the energy density of $u$ considered as a measure on $\R^{n+1}_+$. In the global case $T=\infty$ it provides a homogeneous norm equivalent to the $\mathrm{BMO}$ norm. \\\\
In the half-harmonic case, we adopt a similar strategy. However, instead of the caloric extension, we consider the harmonic extension $v: \R^{n+1}_+ \to \R^m$ given by $v(t)=p_t \ast a$ where $p_t$ is the Poisson kernel for $t>0$. Then
$v$ is the solution to the homogeneous linear problem $\del_t v +(-\Delta)^{\frac 1 2}v=0$ with $v(0)=a$. The fractional energy density is directly derived from the expression for the fractional energy $\mathcal{E}_{\frac 1 2}(v)$ in \eqref{eq:energy}
\begin{equation} \label{eq:fsq}
\absOd{\FG v}^2(x) = \frac{\gamma_n}{2}\int_{\R^n} \frac{|v(x+h)-v(x)|^2}{|h|} \frac{dh}{|h|^n}.
\end{equation}
The notation refers to the off-diagonal norm (see section \ref{sec:frac})
of the fractional gradient
\begin{equation}  \label{eq:frg}
\FG v(x,y) =\sqrt{ \frac{\gamma_n}{2}} \frac{v(x)-v(y)}{|x-y|^{\frac 1 2}} 
\end{equation}
introduced in \cite{Mazowiecka2018} to capture compensated compactness and improved regularity phenomena in fractional harmonic maps. Notably, as observed and used in \cite{Wettstein2022}, the half-harmonic map heat flow equation takes the explicit form
\begin{equation} \label{eq:HHMHF1}
    \del_t u + (-\Delta)^{\frac 1 2} u = \absOd{\FG u}^2 u. 
\end{equation}
This motivates the following semi-norm on initial conditions $a \in L^\infty(\R^n; \R^m)$
\begin{equation} \label{eq:semiA}
[a]_{\mathcal{A}_T}^2:= \sup_{0<r<T}\sup_{x \in \R^n} \int_0^r \fint_{B_r(x)} \absOd{\FG v}^2 
\quad \text{where} \quad v= p_t \ast a 
\end{equation}
for $T \in [0, \infty]$ 
to define scaling-invariant function spaces  
\begin{equation}\label{eq:A}
\mathcal{A} = \{ a \in L^\infty(\R^n; \R^m): [a]_{\mathcal{A}_\infty} < \infty\} 
\end{equation}
and
\begin{equation}\label{eq:V}
\mathcal{V}= \{ a \in \mathcal{A}: \lim_{T \to 0}[a]_{\mathcal{A}_T} =0 \}.
\end{equation}
The spaces $\mathcal{A}$ and $\mathcal{V}$ replace $\mathrm{BMO}$ and $\mathrm{VMO}$, respectively, but strictly include the scaling-invariant Sobolev and Besov spaces, i.e. 
\begin{equation} \label{eq:Sobolev_Besov}
L^\infty \cap \dot{H}^{\frac n 2}(\R^n;\R^m) \subset L^\infty \cap \dot{B}^{\frac n 2}_{2, \infty}(\R^n;\R^m) \subset \mathcal{A}.
\end{equation}
More precisely, $\mathcal{A}$ contains the $Q_{\alpha}$ spaces for $\alpha \ge 0$ as introduced in \cite{EssenJansonPengXiao2000}. 
Those spaces capture fractional mean oscillation and serve as a bridge between $\mathrm{BMO}$ and H\"older-type regularity, though they are not interpolation spaces in the classical sense. A central feature of the $Q_\alpha$
framework is its characterization via Carleson measures $\alpha \in (0,1)$. 
Buiding on this characterization the well-posedness of certain fractional variants of the Navier-Stokes equation in generalized $Q$ spaces have been investigated in \cite{LiZhai2010}.
Notably, the limiting case \( Q_0 \), captured by our results, forms a proper subspace of \( \mathrm{BMO} \), offering a slightly more refined control of oscillation. In particular, we observe that the space \( \mathcal{A} \) permits jump discontinuities in one dimension and point defects in higher dimensions.\\
Our key result establishes solvability of \eqref{eq:HHMHF1} for initial conditions $a:\R^n \to \Sm$ in $\mathcal{A}$ with $[a]_{\mathcal{A}_T}$ sufficiently small up to time $T$ in certain spaces $X_T$ to be defined below. In particular we obtain local existence for $a \in \mathcal{V}$ and global existence if $[a]_{\mathcal{A}}$ is sufficiently small. Moreover, uniqueness holds in the space \( X_T \), with continuous dependence on the initial data.\\
Our approach is primarily based on real space methods, i.e,  on the fractional gradient rather than the fractional Laplacian. The specific choice of the spaces $\mathcal{A}$ and $X_T$ is particularly well-suited to the structure of the nonlinearity in \eqref{eq:HHMHF1} involving \( \FG \). The operator \( \FLquarter \) is less compatible in this context: unlike the fractional gradient, it lacks a clean product structure and does not interact well with the nonlinear term of the PDE. This complicates requisite estimates for the Picard iteration and makes the use of \( \FG \) the more natural choice for the analysis of \eqref{eq:HHMHF1}. Merging fractional calculus and Carleson measure estimates, a central analytic tool is the decomposition of the fractional gradient into local and non-local parts in order to overcome difficulties caused by slow decay in connection with the Poisson integral and Gagliardo-type interpolation
estimates in terms of the half-gradients.

\section{Main Results}

\subsection{Function spaces}

Based on the notion of fractional gradients \eqref{eq:fsq} we define, for measurable functions $u:\R^n \times (0,T) \to \R^m$, the seminorms 
\begin{align*}
    \snorm{u}{X_{T}^{(0)}} &= \esssup_{\substack{0 < t < T \\ x \in \R^n}} t^{\frac{1}{2}} \absOd{\FG u}(x,t) \\ &+ \esssup_{\substack{0 < t < T \\ x \in \R^n}} \left(\int_0^t \fint_{B_t(x)} \absOd{\FG u}^2(y,s) \dy \ds\right)^{\frac{1}{2}}\\
    \snorm{u}{X_{T}^{(1)}} &= \esssup_{\substack{0 < t < T \\ x \in \R^n}} t^{\frac{3}{2}} \absOd{\FG \nabla u}(x,t) \\
    &\quad + \esssup_{\substack{0 < t < T \\ x \in \R^n}} \left(\int_0^t \fint_{B_t(x)} s^2 \absOd{\FG \nabla u}^2(y,s) \dy \ds\right)^{\frac{1}{2}}
\end{align*}
and the norm
\[
\tnorm{u}{X_T} = \norm{u}{X_T} + \snorm{u}{X_T} \quad \text{where} \quad  \norm{u}{X_T} = \esssup_{\substack{0 < t < T \\ x \in \R^n}} \abs{u(x,t)}
\]
giving rise to Banach spaces 
\[
    X_T = \{u : (0,T) \times \R^n \to \R^m \text{ measurable} \mid \tnorm{u}{X_T} < \infty\}
\]
and
\[
   X_{0,T} = \{u \in X_T \mid \lim_{R \to 0}\snorm{u}{X_R} = 0\}.
\]
The translation and dilation symmetry for half-harmonic map heat flows
\[
u(x,t) \mapsto u(x_0+\lambda x, \lambda t) 
\quad \text{for } x_0 \in \R^n  \text{ and }\lambda>0 
\]
defines an isometry between $X_T$ and $X_{T/\lambda}$.
In the context of $\mathcal{A}$ and $\mathcal{V}$ defined in \eqref{eq:A} and \eqref{eq:V}, the spaces occur as extension spaces and serve as solution spaces for \eqref{eq:HHMHF1}. In this framework we obtain the following results:

\subsection{Well-posedness}

\begin{theorem}[Existence]\label{thm:existenceSph}
    There exists $\varepsilon > 0$ such that for all $a \in \mathcal{A}$ with $\snorm{a}{\mathcal{A}_T} < \varepsilon$ and $\abs{a}^2 = 1$, there exists $u \in X_T$ satisfying $\abs{u}^2 = 1$ and solving 
    \begin{align}\label{eq:HHMHF}
        \begin{cases}
            \partial_t u = \FLhalf u + u \absOd{\FG u}^2, \\
            u(0) = a,
        \end{cases}
    \end{align}
    with initial data is attained in the sense of Schwartz distributions. 
    In particular, for all $a \in \mathcal{V}$, there exists some $T > 0$ and a solution $u \in X_{0,T}$ with $u(0)=a$. 
\end{theorem}

The notion of a solution refers to the mild formulation of \eqref{eq:HHMHF} which turns out to be equivalent to the
weak formulation, see Proposition \ref{prop:MildWeak}. Certain regularity properties can directly be deduced from the equation. 
Since \(u \in X_T\), we infer that 
    \[
    \FLhalf u, \, \absOd{\FG u}^2 \in L^\infty((\varepsilon, T) \times \R^n)
    \]
 and hence $\partial_t u \in L^\infty((\varepsilon, T) \times \R^n)$ for all \(\varepsilon > 0\). Thus, we obtain
    \begin{equation} \label{eq:cont}
        u \in      C^0((0, T) \times \R^n; \Sm).
    \end{equation}
Moreover, we have continuity of the map \(R \mapsto \snorm{u}{X_{R}^{(0)}}\) which is an important element in proof of Theorem~\ref{thm:ContDep} below.
Finally, upgrading the fixed point argument to higher order norms $X_{T}^{(k)}$ one can expect smoothness $u \in C^\infty((0, T) \times \R^n; \Sm)$ similarly to the case of the Navier–Stokes equation addressed in~\cite{Germain2007}.

\begin{theorem}[Uniqueness]\label{thm:uniqueness}
    There exists $\delta > 0$ such that solutions $u_1, u_2 \in X_T$ to~\eqref{eq:HHMHF} with $u_1(0) = u_2(0) \in \mathcal{A}$ and 
    \begin{align*}
        \limsup_{R \searrow 0}\snorm{u_i}{X_R} \leq \delta \quad \text{for } i=1,2
    \end{align*}
    satisfy $u_1 \equiv u_2$ in $X_T$. 
\end{theorem}

In particular, uniqueness holds true for solutions in $X_{0,T}$ and small solutions in $X_T$ as constructed in Theorem \ref{thm:existenceSph}. Theorem \ref{thm:uniqueness}, however, requires smallness only near $t=0$.

\begin{theorem}[Continuous dependence on initial data]\label{thm:ContDep}
    Let $a \in \mathcal{V}$ and suppose that $(a_n)_{n \in \N}\subset \mathcal{V}$ satisfies 
    \[
    \lim_{n \to \infty} \left( \norm{a_n - a}{L^\infty} +\snorm{a_n - a}{\mathcal{A}_\infty}\right) = 0.
    \]
    Then there exists $T > 0$ such that the corresponding unique solutions $u_n$ and $u$ belong to $X^0_{T}$ and satisfy 
    \[
    \lim_{n \to \infty} \left( \norm{u_n-u}{L^\infty} + \snorm{u_n - u}{X_T^{(0)}} \right) =  0.
    \]
\end{theorem}


\subsection{Relation with $Q$ and Besov spaces} 
The results extend the local and global well-posedness results 
\cite{KK2024, MelcherSakellaris} for initial data 
$a:\R^n \to \mathbb{S}^{m-1}$ in the critical Sobolev class $H^{n/2}$ which is contained in $\mathcal{V}$.
For maps from the unit circle $\mathbb{S}^1$, local well-posedness in the energy space $H^{1/2}$ has been obtained before in \cite{Wettstein2022, Wettstein2023, Struwe2024}, based on fractional calculus and extension via the Dirichlet-to-Neumann map, respectively, as part of a program on almost smooth global weak solutions. Our approach starts from fractional calculus and uses extension ideas in the spirit of maximal functions to capture rougher initial data that e.g exhibits some homogeneous structure. To this end 
we recall the notion of \( Q_\alpha \)-spaces defined for \(\alpha \in \R\) as the set of measurable functions \( f \) satisfying
\begin{align*}
    \snorm{f}{Q_\alpha}^2 := \sup_{\substack{x \in \R^n \\ r > 0}} \frac{1}{r^{n-2\alpha}} \int_{B_r(x)} \int_{B_r(x)} \frac{\abs{f(y) - f(z)}^2}{\abs{y-z}^{n+2\alpha}} \dy \dz < \infty.
\end{align*}
The function spaces $Q_\alpha(\R^n)$ were introduced and studied in \cite{EssenJansonPengXiao2000}.These spaces form a nested hierarchy: if $\alpha < \beta$, then $Q_\beta(\R^n) \subseteq Q_\alpha(\R^n)$. Moreover, they contain only constant functions when if $\alpha \ge 1$ and $n \ge 2$, or when $\alpha>1/2$ and $n=1$. In the case
 $\alpha < 0$, we have the identification $Q_\alpha(\R^n) = \mathrm{BMO}(\R^n)$. In the borderline case $\alpha=0$, we establish the embedding of $L^\infty \cap Q_0(\R^n)$ into our initial data space $\mathcal{A}$:
 
 \begin{theorem}\label{thm:Q0Included}
    There is constants $c>0$ such that  
    \begin{align*}
        \snorm{a}{\mathcal{A}_\infty} \leq  \snorm{a}{Q_0} \quad \text{for all} \quad a \in Q_0(\R^n).
    \end{align*}
\end{theorem}

We have $Q_0(\R^n) \subsetneq \mathrm{BMO}(\R^n)$ (cf.  \cite{EssenJansonPengXiao2000} Remark 2.11), but we capture relevant cases of rough date
and a true extension of critical Sobolev well-posedness. In fact, from the embedding properties of $Q_\alpha$ spaces into homogeneous Besov spaces, also examined in \cite{EssenJansonPengXiao2000} Theorem 2.7, in particular,
see e.g. \cite{BerghLoefstroem, Yuan2010},
\[
\dot{B}^{\frac{n}{2}}_{2,\infty} (\R^n) \subset \dot{B}^{\frac{1}{2}}_{2n,\infty} (\R^n) \subset Q_0(\R^n), 
\]
we deduce well-posedness results in terms of this critical Besov norm weaker than the critical Sobolev norm $H^{n/2}$. In particular, for $n = 1$, taking into account the embedding properties of functions of bounded variation as outlined in \cite{ChoksiKohnOtto, CohenDahmenDauberchiesDeVore} 
\[
BV(\R)  \subset L^\infty \cap \dot{B}^{\frac{1}{2}}_{2,\infty} (\R),
\]
allows us to handle initial data that may include small jump discontinuities. Note, however, that for $a=a_0 \chi_{(-\infty, 0)} +a_1 \chi_{(0,\infty)} \in \mathcal{A}$ with constants $a_0,a_1 \in \mathbb{S}^{m-1}$, $v=p_t \ast a$ is self-similar, and hence the averages in
\eqref{eq:semiA} are independent of $r$. We infer that $a \notin \mathcal{V}$. 

\subsection{Homogenous initial data and self-similar expander}


In higher dimensions $n \ge 2$ our theory covers homogeneous initial data of the form 
\begin{equation} \label{eq:homogeneous_data}
a(x) = \varphi\!\left(\frac{x}{|x|}\right), \quad x \in \R^n \setminus \{0\},
\end{equation}
i.e. $a=\varphi \circ h$ where $\varphi:\mathbb{S}^{n-1} \to \mathbb{S}^{n-1}$ is a Lipschitz map and $h(x)=x/|x|$. Clearly, 
\[     
\snorm{a}{Q_0} \le  \mathrm{Lip}(\varphi) \snorm{h}{Q_0} \quad \text{while} \quad 
\snorm{h}{Q_0} \lesssim \snorm{h}{\dot{B}^{\frac n 2}_{2, \infty}}< \infty. \]
Indeed, applying the Fourier transform 
to $\eta(x) = \abs{x}$ and then to $\nabla \eta = h$ yields
\begin{align*}
    \mathcal{F}(h) = \text{p.v.} \left(   \frac{c_n \xi}{\abs{\xi}^{n+1}} \right),
\end{align*}
see \cite{Grafakos2014a}. Passing to the Paley-Littlewood decomposition, we obtain the estimate
\begin{align*}
    \norm{h * \psi_k}{L^2} \leq c 2^{-\frac{n k}{2}}
\end{align*}
for the $k$-th homogeneous dyadic block $\psi_k$, and hence $h \in \dot{B}^{\frac{n}{2}}_{2,\infty}(\mathbb{R}^n)$. 
Therefore, Theorem \ref{thm:existenceSph} and \ref{thm:uniqueness} apply for $\mathrm{Lip}(\varphi)$ sufficiently small and provide a unique global solution
$u$. However, since $v=p_t \ast a$ is self-similar, the $\mathcal{A}_T$ norm cannot decay, and we conclude that $a \not\in \mathcal{V}$. Moreover, scaling 
and uniqueness imply that $u(x,t) =u(\lambda x, \lambda t)$ for all
$(x,t) \in \R^n \times (0,\infty)$ and $\lambda>0$, and letting 
$\lambda = \frac{1}{t}$, we conclude as in \cite{DeLaire}:

\begin{theorem}[Self-similar expanders]
    There exist a constant $\ell = \ell(n)$ such that, if $a$ is of the form \eqref{eq:homogeneous_data} and $\mathrm{Lip}(\varphi) \leq \ell$, then the unique solution $u$ of \eqref{eq:HHMHF}, emanating from $a$, has the form
    \begin{equation}
        u(x,t) = u\left(\frac{x}{t},1\right) \quad  \text{ for all } (x,t) \in \R^n \times (0,\infty).
    \end{equation}
\end{theorem}



\section{Preliminaries}

\subsection{Poisson kernel and fractional heat equation}
Recall that the Poisson kernel is given by   
\begin{align*}
 p_t(x) = \frac{1}{t^n} p\left(\frac{x}{t}\right)
 \quad \text{where} \quad    p(x) = \frac{\Gamma(\frac{n+1}{2})}{\pi^{\frac{n+1}{2}}} \frac{1}{(1 + \abs{x}^2)^{\frac{n+1}{2}}}
\end{align*}
It defines the semigroup \( \left(\semiGroup{t}\right)_{t \geq 0} \) generated by \(\FLhalf\), which acts via convolution
\begin{align*}
   ( \semiGroup{t} f) (x) = p_t * f (x) = \int_{\R^n} p_t(x - y) f(y) \dy
\end{align*}
and gives rise to a mild formulation for the initial value problem for the inhomogeneous linear fractional heat equation
\[
\del_t u + (-\Delta)^{1/2} u =f  \quad \text{with} \quad u(0)=a.
\] 
We also require a concept of weak solutions.
Accounting for the limited decay properties, we introduce the test space $\mathcal{T}$ consisting of all functions $\varphi \in C^\infty((0,T)\times\mathbb{R}^n)$ whose Fourier transforms can be represented as
    \[
        \hat{\varphi}(\xi,t) = f(|\xi|,t)\psi(\xi,t)
    \]
    where $f \in C^\infty(\mathbb{R}\times(0,T))$ and $\psi \in C^\infty(\mathbb{R}^n\times(0,T))$ satisfy:
    \begin{itemize}
        \item[(a)](temporal compactness) There exists a compact set $K \subset (0,T)$ such that
        \[
            \hat{\varphi}(\xi,t) = 0,\quad \text{for } t \notin K.
        \]
        \item[(b)] (spatial decay and smoothness) For any $\gamma \in C_c^\infty(\mathbb{R}^n)$ with $\gamma = 1$ near $0$
        \[
            (1-\gamma(\xi))\hat{\varphi}(\xi,t) \in \mathcal{S}((0,T)\times\mathbb{R}^n).
        \]
    \end{itemize}
The motivation for $\mathcal{T}$ is stability under relevant operations and spatial decay such as the Poisson kernel. Specifically, for any $\varphi \in \mathcal{T}$, we have the following:
\begin{align}
    \FLhalf \varphi &\in \mathcal{T}, \\
    p_t * \varphi &\in \mathcal{T}, \\
    \abs{\varphi(x)} &\leq \frac{c}{1+\abs{x}^{n+1}}. \label{eq:growthT}
\end{align}
It is clear that $\mathcal{T} \subset S_{\frac{1}{2}}$, as defined, for instance, in Lemma 12.2 of~\cite{Stinga2024}. The space $S_{\frac{1}{2}}$ provides a well-established framework for distributions related to the fractional Laplacian. However, while $S_{\frac{1}{2}}$ is constructed via the pointwise singular integral definition of the fractional Laplacian, the space $\mathcal{T}$ is explicitly defined through a Fourier space structure. This Fourier-based characterization renders $\mathcal{T}$ particularly suitable for our context, especially when dealing with operators such as the Paley–Littlewood trucations $P_\varepsilon$ projecting onto frequencies of size $\gtrsim 1/\varepsilon$. \\\\
Note that \eqref{eq:growthT} follows from applying the Fourier transform to~$\varphi$ and assuming without loss of generality that~$\psi \in C_c^\infty(\R^{n+1})$ with $f(0,t)=\partial_rf(0,t)=0$. Using Lemma~12.2 from~\cite{Stinga2024}, we obtain
\begin{align*}
    \abs{\FLhalf \hat{\psi}(x,t)} \leq \frac{c}{1+\abs{x}^{n+1}},
\end{align*}
which confirms the desired growth condition by integration by parts applied to~$\hat{\varphi}$.

\begin{proposition}\label{prop:MildWeak}
    Let $a \in L^1_{\mathrm{uloc}}(\mathbb{R}^n)$, and $f \in L^1_{\mathrm{uloc}}(\mathbb{R}^n; L^1(0,T))$. Then the following are equivalent:
    \begin{itemize}
        \item[(i)] $u$ is a \emph{mild solution} of the inhomogeneous fractional half-heat equation:
        \begin{align*}
            u(t) = \semiGroup{t} a + \int_0^t \semiGroup{t-s} f(s) \, ds.
        \end{align*}
        \item[(ii)]  $u \in L^\infty((0,T),L^1_{\mathrm{uloc}}(\mathbb{R}^n))$ is a \emph{weak solution} with $u(0) = a$ attained in the sense of Schwartz distributions, satisfying
        \begin{align*}
            \int_{\mathbb{R}^n \times (0,T)} \left( u \, \partial_t \varphi + u \, \FLhalf \varphi \right) \, dx \, dt
            =
            \int_{\mathbb{R}^n \times (0,T)} f \, \varphi \, dx \, dt,
        \end{align*}
        for all $\varphi \in \mathcal{T}$.
    \end{itemize}
    In either case we have that $u$ also attains its initial data in $L^1_{loc}(\R^n)$. \\
\end{proposition}
The proof of Proposition \ref{prop:MildWeak} is an adaptation of Theorem~11.2 from~\cite{LemariePierre2002}. The main difference lies in the differentiation process in the sense of distributions, which is handled by considering the space~$\mathcal{T}$. With this approach, all analogous calculations remain valid, and particularly, $\varphi * p_t \in \mathcal{T}$ if $\varphi \in \mathcal{S}(\R^n \times (0,T))$ has compact support in time.\\

When dealing with uniformly locally $p$-integrable functions we require the 
following extension of  he Hausdorff–Young inequality whose proof follows directly from Minkowski's inequality:

\begin{lemma}\label{lem:HausdorfYoungUloc}
    Given $f \in L^1(\R^n)$ and $g \in L^p_{\mathrm{uloc}}(\R^n)$ for some $p \in [1,\infty]$, it holds that $f * g \in L^p_{\mathrm{uloc}}(\R^n)$ with the estimate:
    \begin{align*}
        \norm{f * g}{L^p_{\mathrm{uloc}}} \leq \norm{f}{L^1} \norm{g}{L^p_{\mathrm{uloc}}}.
    \end{align*}
    Similarly, for $T > 0$, $f \in L^1(\R^n)$ and $g \in L^p_{\mathrm{uloc}}(\R^n \times [0,T])$, we have
    \begin{align*}
        \norm{f * g}{L^p_{\mathrm{uloc}}(\R^n \times [0,T])} \leq \norm{f}{L^1(\R^n)} \norm{g}{L^p_{\mathrm{uloc}}(\R^n \times [0,T])}.
    \end{align*}
\end{lemma}

\subsection{Decomposition of fractional gradients}\label{sec:frac}

For two measurable functions \( F, G :\R^n \times \R^n \to \R^m \), we define \( \scp{\cdot}{\cdot}_{od} \) as  
\[
\scp{F}{G}_{od} := \int_{\R^n} \frac{F(x, y) G(x, y)}{|x - y|^n} \dy.
\]
This operator acts as a non-local scalar product for vector fields. Thus, we define the squared modulus of a vector field \( F\) as  
\[
\absOd{F}^2 := \scp{F}{F}_{od}.
\]
For \( p \in [1,\infty] \), this leads to the definition of \( L^p_{od}(\R^n \times \R^n; \R^m) \), which contains any measurable function \( F : \R^n \times \R^n \to \R^m \) such that  
\[
\norm{F}{L^p_{od}}^p := \int_{\R^n} \absOd{F}^p(x) \dx < \infty
\]
with the obvious modification for \( p = \infty \). This modulus contains local and non-local behaviour. To differentiate these better, we define for any measurable function \( F :\R^n \times \R^n \to \R^m \) and \( 0 < r_0 < r_1  \)  
\[
\absOd{F}^{(r_0,r_1)}(x) := \left(\int_{B_{r_1}(0) \setminus B_{r_0}(0)} \frac{\abs{F(x + h, x)}^2}{\abs{h}^n} \dh\right)^{\frac{1}{2}}.
\]
An interesting application of this modified modulus is that for \( u \in W^{1,\infty}(\R^n) \), we obtain for the fractional gradient norm
defined in \eqref{eq:fsq} and \eqref{eq:frg}
\[
\absOd{\FG u} \leq \absOd{\FG u}^{(0,1)} + \absOd{\FG u}^{(1,\infty)} \leq c \norm{\nabla u}{L^\infty} + c\norm{u}{L^\infty}.
\]
Scaling this result gives
\[
\norm{\FG u}{L^\infty_{od}} \leq c \norm{u}{L^\infty}^{\frac{1}{2}} \norm{\nabla u}{L^\infty}^{\frac{1}{2}}.
\]
One useful inequality for fractional derivatives is Minkowski's integral inequality, which states:
\begin{align}
    \left( \int_{\Omega} \abs{\int_{\Sigma} f(x,y) \, d\mu(y)}^p \, d\nu(x) \right)^{\frac{1}{p}}
    \leq \int_{\Sigma} \left( \int_{\Omega} \abs{f(x,y)}^p \, d\nu(x) \right)^{\frac{1}{p}} d\mu(y),
\end{align}
for any measurable function $f: \Omega \times \Sigma \to \R$, and non negative measures $\mu$ and $\nu$. \\
Using this inequality, it follows immediately that:
\begin{align}
    \absOd{\FG u * v}^{(r_0,r_1)}(x) \leq \abs{v} * \absOd{\FG u}^{(r_0,r_1)}(x).
\end{align}

\begin{lemma}\label{lem:growthEstimate}
    For a function $u \in C^1(\R^n;\R^m)$ with
    \begin{align*}
        \abs{\nabla u(x)} \leq \frac{c}{1+\abs{x}^{n+\delta}}
    \end{align*}
    for some $\delta>0$, the follwoing estimates hold true:
    \begin{align*}
        \absOd{\FG u}^{(0, \frac{\abs{x}}{4})}(x) \leq \frac{c^\prime}{1+\abs{x}^{n+\delta-\frac{1}{2}}}
    \end{align*}
    and 
    \begin{align*}
        \absOd{\FG u}^{(0, 1)}(x) \leq \frac{c^\prime}{1+\abs{x}^{n+\delta}}.
    \end{align*}
    In particular, if $u \in C^1(\R^n \times [0,T];\R^m)$ is such that
    \begin{align*}
        \sup_{s \in [0,T]} \abs{\nabla u(x,s)} \leq \frac{c}{1+\abs{x}^{n+\delta}}
    \end{align*}
    the same estimates holds uniformly.
\end{lemma}

\begin{proof}
    We follow an approach similar to the growth estimates for the fractional Laplacian in Lemma 12.2 of~\cite{Stinga2024}. 
    From the mean value theorem we obtain a $\xi = (x+h) t + (1-t)x$, for a $t = t(x,h) \in [0,1]$ such that
    \begin{align*}
        \nabla u (\xi)h = u(x+h) - u(x). 
    \end{align*}
    Rearranging leads to $x = \xi - th$, thus $\abs{x} \leq \abs{\xi} + \frac{\abs{x}}{4}$, i.e. $\abs{x} \leq \frac{4}{3} \abs{\xi}$. Hence
    \begin{align*}
        \abs{\nabla u (\xi)h} \leq \abs{h} \frac{c}{1+\abs{\xi}^{n+\delta}} \leq \abs{h} \frac{\tilde{c}}{1+\abs{x}^{n+\delta}}.
    \end{align*}
    With this estimate we conclude
    \begin{align*}
        \int_{B_{\frac{\abs{x}}{4}}(0)} \frac{\abs{u(x+h) - u(x)}^2}{\abs{h}^{n+1}} \dh \leq \left(\frac{c}{1+\abs{x}^{n+\delta}}\right)^2 \int_{B_{\frac{\abs{x}}{4}}(0)} \frac{1}{\abs{h}^{n-1}} \dh = c \left(\frac{c}{1+\abs{x}^{n+\delta}}\right)^2 \abs{x}.
    \end{align*}
    Hence we have the estimate
    \begin{align*}
        \absOd{\FG u}^{(0, \frac{\abs{x}}{4})}(x) \leq \frac{c^\prime}{1+\abs{x}^{n+\delta-\frac{1}{2}}}.
    \end{align*}
    The second estimate is proven analogously.
\end{proof}

The motivation for this lemma stems from slow decay arising from the non-local nature of the fractional gradient. A concrete example is given by the Poisson kernel in dimension \( n = 1 \), where one computes
\begin{align}\label{eq:poissoinSchlecht}    
\absOd{\FG p}^2(x) = \frac{1}{\pi} \cdot \frac{1}{2x^2 + 2},
\end{align}
showing that \( \absOd{\FG p} \notin L^1(\R) \), even though \( p^{(k)} \in L^q(\R) \) for all \( q \in [1, \infty], \,k \in \N_0\).

\subsection{Interpolation inequalities involving fractional gradients}

\begin{lemma}\label{lem:BesovInterpolation}
    Let \( k \in \N \). Then, for any \( 0 < \varepsilon < \frac{1}{2} \), we have the estimate
    \[
        \norm{u}{B^{k + \varepsilon}_{\infty, \infty}} \leq C \norm{u}{L^\infty} + C \frac{2}{\varepsilon R^\varepsilon} \norm{\nabla^k u}{L^\infty} + C \frac{R^{\frac{1}{2} - \varepsilon}}{\sqrt{1 - 2\varepsilon}} \norm{\FG \nabla^k u}{L^\infty_{od}}.
    \]
for a constant $C>0$ that is independent of $\eps$.
\end{lemma}


\begin{proof}
    Without loss of generality, we assume \( k = 1 \), where 
    \[
        \norm{u}{B^{1 + \varepsilon}_{\infty, \infty}} \sim \norm{u}{L^\infty} + \snorm{\nabla u}{\dot{B}^{\varepsilon}_{\infty, \infty}}.
    \]
    Writing \(v = \nabla u\) we obtain for the $j$-th dyadic block 
    a pointwise bound
    \[
        \abs{\Delta_j v(x)} \leq \int_{\R^n} \abs{\psi_j(x-y)(v(y) - v(x))} \dy
    \]
    \[
        \leq 2^{j(n-n-\varepsilon)}\int_{\R^n} \frac{1}{\abs{h}^{n+\varepsilon + \delta - \delta}}\abs{v(x+h) - v(x)} \dh
    \]
    \[
        \leq 2^{-j \varepsilon}\left(\int_{B_R(0)} \frac{1}{\abs{h}^{n-2\delta}} \dh \right)^{\frac{1}{2}} \left(\int_{B_R(0)} \frac{\abs{v(x+h)-v(x)}^2}{\abs{h}^{n+2(\varepsilon + \delta)}} \dh \right)^{\frac{1}{2}}
    \]
    \[
        +  2^{-j \varepsilon} 2 \norm{v}{L^\infty} \int_{\R^n \setminus B_R(0)} \frac{1}{\abs{h}^{n+\varepsilon}} \dh.
    \]
    Letting \( \delta = \frac{1}{2} - \varepsilon \), this implies
    \[
        \snorm{v}{\dot{B}^{\varepsilon}_{\infty, \infty}} \leq \frac{R^\delta}{\sqrt{2 \delta}} \norm{\FG v}{L^\infty_{od}} + \frac{2}{\varepsilon R^\varepsilon} \norm{v}{L^\infty}.
    \]
    Taking into account that for \( k \in \N \)
    \[
        \snorm{u}{\dot{B}^{k + \varepsilon}_{\infty, \infty}} \sim \norm{u}{L^\infty} + \snorm{\nabla^k u}{\dot{B}^{\varepsilon}_{\infty, \infty}}
    \]
  we infer the general result. 
\end{proof}

From scaling arguments we finally obtain the following interpolation inequalities:
\begin{align}
    \norm{\nabla^k u}{L^\infty} 
    &\leq c \, \norm{u}{L^\infty}^{\frac{1}{2k + 1}} 
    \, \norm{\FG \nabla^k u}{L^\infty_{od}}^{\frac{2k}{2k + 1}}, \\
    \norm{\FLhalf u}{L^\infty} 
    &\leq c \, \norm{u}{L^\infty}^{\frac{1}{2}} 
    \, \norm{\FG \nabla u}{L^\infty_{od}}^{\frac{1}{2}}. \label{eq:nonsnorminterpolationLaplace}
\end{align}
Adapting the proof of Lemma~\ref{lem:BesovInterpolation}, we observe that for $\rho \in C_c^\infty(B_1(0))$ and any measurable function $v$,
\begin{align*}
    \abs{\int_{\R^n} (v(x) - v(y))\rho(x-y) \dy} 
    \leq c \norm{\rho}{L^\infty} \absOd{\FG v}^{(0,1)}.
\end{align*}
Applying this to $v = \nabla u$ and $v = u$, we obtain
\begin{align*}
    \abs{\nabla u}(x) 
    &\leq \abs{\nabla u(x) - (\nabla u * \rho)(x)} 
    + \abs{(u * \nabla\rho)(x)} \\
    &\leq c\left(\absOd{\FG u}(x) + \absOd{\FG \nabla u}(x)\right).
\end{align*}
Using a scaling argument, this leads to another interpolation inequality
\begin{align}
    \norm{\nabla u}{L^p} 
    \leq c \norm{\FG u}{L^p}^{\frac{1}{2}}
    \norm{\FG \nabla u}{L^p}^{\frac{1}{2}}, \qquad \text{for all } p \in [1,\infty].\label{eq:snorminter}
\end{align}

\section{Linear estimates}

\subsection{Homogenous estimates}\label{subsub:homogenousPart}

In the following we examine solutions $v \in X_T$ to the homogeneous problem
\[
\del_t v+ (-\Delta)^{\frac 1 2} v =0 \quad \text{with } v(0)=a
\]
for some $a \in \mathcal{A}$ given by $v := p_t * a$.

\begin{lemma}\label{lem:standardEstimate}
    For $a \in \mathcal{A}$, the following estimate holds with a constant $c_1>0$ that only depends on the dimension:
    \begin{align*}
        \sup_{0 < t < T} t \norm{\FG (p_t * a)}{L^\infty_{od}}^2 \leq c_1\snorm{a}{A_T}^2.
    \end{align*}
\end{lemma}

\begin{proof}
Due to scaling and translation it suffices to show the inequality
\begin{align*}
    \absOd{\FG (p_1 * a)}(0) 
    \leq c \left( \sup_{x \in \R^n} \int_0^1 \int_{B_1(x)} \absOd{\FG (p_s * a)}^2(y) \, dy \, ds \right)^{\frac{1}{2}}.
\end{align*}
Since we have
\begin{align*}
    v(\cdot,1) = p_1 * \alpha = p_{1-s} * p_s *a, \quad \text{ for all } s \in \left( 0, \frac{1}{2} \right],
\end{align*}
applying the fractional gradient yields:
\begin{align*}
    \absOd{\FG v}(\cdot,1) 
    \leq \left( p_{1-s} \absOd{\FG v(s)} \right)(\cdot), 
    \quad \text{ for all } s \in \left(0, \frac{1}{2} \right].
\end{align*}
Averaging over $[0,\frac{1}{2}]$ yields:
\begin{align*}
    \absOd{\FG v}(0,1) 
    \leq c \int_0^{\frac{1}{2}} \int_{\R^n} p_{1-s}(y) \absOd{\FG v(s)}(y) \, dy \, ds 
    \leq c \int_0^1 \int_{\R^n} \frac{\absOd{\FG v(s)}(y)}{1 + \abs{y}^{n+1}} \, dy \, ds.
\end{align*}
Decomposing $\R^n$ into annuli $A_k = B_k \setminus B_{k-1}$, we apply Jensen's inequality:
\begin{align*}
    \int_0^{1} \int_{\R^n} \frac{\absOd{\FG v(s)}(y)}{1 + \abs{y}^{n+1}} \, dy \, ds 
    &\leq c \sum_{k = 1}^\infty \frac{1}{k^2} \sup_{x \in \R^n} \int_0^1 \fint_{B_1(x)} \absOd{\FG v} \, dy \, ds \\
    &\leq c \left( \sup_{x \in \R^n} \int_0^1 \fint_{B_1(x)} \absOd{\FG (p_s * a)}^2(y) \, dy \, ds \right)^{\frac{1}{2}}.
\end{align*}
\end{proof}

\begin{lemma}\label{lem:initialCond} 
 For $a \in \mathcal{A}$, the following estimate holds with a constant $c_2>0$ that only depends on the dimension:
\[
\snorm{\semiGroup{t}a}{X_T} \leq c_2 \snorm{a}{\mathcal{A}_T}.
\]
    \end{lemma}

\begin{proof}
    The seminorm of $X_T$ consists of four parts. The \(\snorm{\cdot}{X_T^{(0)}}\) part follows directly from Lemma~\ref{lem:standardEstimate}. For the remaining two, we observe the following estimate:
    \begin{align*}
        \absOd{\FG \nabla(p_t * a)} 
        &= \absOd{\FG (\nabla p_{\frac{t}{2}}) * p_{\frac{t}{2}} * a} \notag \\
        &\leq \abs{\nabla p_{\frac{t}{2}}} * \absOd{ \FG (p_{\frac{t}{2}} * a)}.
    \end{align*}
    From this bound, we deduce:
    \begin{align*}
        t^{\frac{3}{2}} \norm{\FG \nabla(p_t * a)}{L^\infty_{od}} 
        &\leq t^{\frac{3}{2}} \norm{\nabla p_{\frac{t}{2}}}{L^\infty} \norm{\FG (p_{\frac{t}{2}} * a)}{L^\infty_{od}} \notag \\
        &\leq c \snorm{a}{\mathcal{A}_T}.
    \end{align*}
    The integral part follows similarly to the proof of Lemma \ref{thm:Q0Included}:
    \begin{align*}
        \sup_{x \in \R^n} \int_0^T \fint_{B_T(x)} s^2 \absOd{\FG \nabla(p_s * a)}^2 \, ds \, dy 
        &\leq c \sup_{x \in \R^n}\int_0^T \fint_{B_T(x)} \absOd{\FG (p_s * a)}^2 \, ds \, dy \notag \\
        &\leq c\snorm{a}{\mathcal{A}_T}^2.
    \end{align*}
\end{proof}

\subsection{Inhomogeneous Part}
In the following we examine solutions $w \in X_T$ to the homogeneous problem
\[
\del_t w+ (-\Delta)^{\frac 1 2} w =f \quad \text{with } w(0)=0
\]
for $f$ in a suitable class $Y_T$ that enjoys suitable mapping properties 
of the corresponding Duhamel operator 
\begin{align}
    G(f)(t) = \int_0^t \semiGroup{t-s}f(s) \ds. 
\end{align}
To this end we introduce the seminorm 
\[
   \snorm{f}{Y_T} = \esssup_{\substack{0 < t < T \\ x \in \R^n}} t^2 \abs{\nabla f(x,t)}
    + \esssup_{\substack{0 < t < T \\ x \in \R^n}} \int_0^t \fint_{B_t(x)} s \abs{\nabla f(y,s)} \dy \ds.
\]
and norm $\tnorm{f}{Y_T}= \norm{f}{Y_T} + \snorm{f}{Y_T}$ where
\[
\norm{f}{Y_T} = \esssup_{\substack{0 < t < T \\ x \in \R^n}} t \abs{f(x,t)} + \esssup_{\substack{0 < t < T \\ x \in \R^n}} \int_0^t \fint_{B_t(x)} \abs{f(y,s)} \dy \ds.
\]
defining the auxilliary space
\[
    Y_T = \{u : (0,T) \times \R^n \to \R^m \text{ measurable} \mid \tnorm{u}{Y_T} < \infty\}.
 \]

\begin{proposition}\label{prop:XTYTEstimate}
    Given $T \in [0,\infty]$ and $f \in Y_T$, then $G(f) \in X_T$ with
    \begin{align}
        \tnorm{G(f)}{X_T} \leq c\tnorm{f}{Y_T},
    \end{align}
    where the constant $c>0$ depends only on the dimension.
\end{proposition}$ $ 

The proof extends the inhomogeneous estimate in \cite{Wang2010} to the fractional setting.
Through scaling and translation, it suffices to consider only the cases \(T = 1\) and \(x = 0\)
in the individual estimates. Assuming \( f \in Y_1 \), we have obtain the following bounds:

\begin{lemma}   $\displaystyle{    \abs{G(f)(0,1)} \leq \tnorm{f}{Y_1}}$.
\end{lemma}

\begin{proof}
    We split the integral into two parts:
    \begin{align*}
        G(f)(0,1) &= \int_{0}^{1} \int_{\R^n} p(y,1-s)f(y,s) \dy \ds \\
        &= \int_{\frac{1}{2}}^{1} \int_{\R^n} p(y,1-s)f(y,s) \dy \ds \\
        &\quad + \int_{0}^{\frac{1}{2}} \int_{\R^n} p(y,1-s)f(y,s) \dy \ds \\
        &= I_1 + I_2.
    \end{align*}
    We estimate each term separately:
    \begin{align*}
        \abs{I_1} &\leq \left( \esssup_{\frac{1}{2} \leq s \leq 1} \norm{f(s)}{L^\infty} \right) 
        \left( \int_{\frac{1}{2}}^{1} \int_{\R^n} p(y,1-s) \dy \ds \right) \leq C \tnorm{f}{Y_1}.
    \end{align*}
    \begin{align*}
        \abs{I_2} &\leq \int_{0}^{\frac{1}{2}} \int_{\R^n} p(y,1-s) \abs{f(y,s)} \dy \ds \\
        &\leq c \int_{0}^{\frac{1}{2}} \int_{\R^n} \frac{1}{1+\abs{y}^{n+1}} \abs{f(y,s)} \dy \ds \\
        &\leq c \left( \sum_{k=1}^{\infty} \frac{1}{k^2} \right) 
        \left( \sup_{y \in \R^n} \int_0^1\fint_{B_1(x)} \abs{f(y,s)} \dy \ds \right) \leq c \tnorm{f}{Y_1}.
    \end{align*}
\end{proof}

\begin{lemma} $\displaystyle{\absOd{\FG G(f)}(0,1) \leq \tnorm{f}{Y_1}}$.
\end{lemma}

\begin{proof}
    Splitting the gradient into two parts:
    \begin{align*}
        \absOd{\FG G(f)}(0,1) \leq \absOd{\FG^{(0,1)} G(f)}(0,1) + \absOd{\FG^{(1,\infty)} G(f)}(0,1).
    \end{align*}
    The second summand can be easily estimated through:
    \begin{align*}
        \absOd{\FG^{(1,\infty)} G(f)}^2(0,1) \leq c \norm{G(f)(\cdot,1)}{L^\infty}^2.
    \end{align*}
    We split the first summand again up $\absOd{\FG^{(0,1)} G(f)}(0,1)$ into multiple parts:
    \begin{align*}
        \absOd{\FG^{(0,1)} G(f)}(0,1) &\leq \int_0^1 \absOd{\FG^{(0,1)} p_{1-s} * f}(0,s) \ds \\
        &\leq \int_0^{\frac{1}{2}} \absOd{\FG^{(0,1)} p_{1-s} * f}(0,s) \ds + \int_{\frac{1}{2}}^1 \absOd{\FG^{(0,1)} p_{1-s} * f}(0,s) \ds.
    \end{align*}
    Using Lemma~\ref{lem:growthEstimate} and since $s \in [0,\frac{1}{2}]$ we can apply a similar argument as before:
    \begin{align*}
        \int_{0}^{\frac{1}{2}}\absOd{\FG^{(0,1)} p_{1-s} * f}(0,s) \ds
        &\leq \int_{0}^{\frac{1}{2}} \absOd{\FG^{(0,1)} p_{1-s}} * \abs{f(\cdot, s)}(0) \ds \\
        &\leq \int_0^{\frac{1}{2}} \int_{\R^n} \frac{c \abs{f(y,s)}}{1+\abs{y}^{n+2}} \dy \ds \\
        &\leq c \left( \sum_{k=1}^{\infty} \frac{1}{k^3} \right) 
        \left( \sup_{y \in \R^n} \int_0^1\fint_{B_1(x)} \abs{f(y,s)} \dy \ds \right) \\
        &\leq c \tnorm{f}{Y_1}
    \end{align*}
    Next, we address the integral
    \begin{align*}
        \int_{\frac{1}{2}}^1 \absOd{\FG^{(0,1)} p_{1-s} * f}(0,s) \ds = \int_{\frac{1}{2}}^1 \frac{1}{\sqrt{1-s}}\absOd{\FG^{(0,\frac{1}{1-s})} m_{1-s}}(0,s) \ds,
    \end{align*}
    where
    \begin{align*}
        p_{1-s} * f(x) = \int_{\R^n} p\left(\frac{x}{1-s} - y\right) f\left((1-s)y, s\right) \dy = m_{1-s}\left(\frac{x}{1-s},s\right).
    \end{align*}
    Splitting the fractional gradient we obtain
    \begin{align*}
        &\int_{\frac{1}{2}}^1 \frac{1}{\sqrt{1-s}}\absOd{\FG^{(0,\frac{1}{1-s})} m_{1-s}}(0,s) \ds 
        \leq \\&\int_{\frac{1}{2}}^1 \frac{1}{\sqrt{1-s}}\absOd{\FG^{(0,1)} m_{1-s}}(0,s) \ds 
        + \int_{\frac{1}{2}}^1 \frac{1}{\sqrt{1-s}}\absOd{\FG^{(1,\frac{1}{1-s})} m_{1-s}}(0,s) \ds.
    \end{align*}
    The second summand can again be estimated in $L^\infty$. The first summand can be estimated through the mean value theorem and the Hausdorff-Young inequality.
\end{proof}

\begin{lemma}
   $\displaystyle{\absOd{\FG \nabla G(f)}(0,1) \leq \tnorm{f}{Y_1}}$.
\end{lemma}

\begin{proof}
    We follow the same decomposition idea
    \begin{align*}
        \nabla G(f) &= \nabla \left(\int_0^1 p_{1-s} * f(s) \ds\right) \\
        &= \int_0^{\frac{1}{2}} \nabla (p_{1-s}) * f(s) \ds + \int_{\frac{1}{2}}^1 p_{1-s} * \nabla f(s) \ds,
    \end{align*}
    and proceed as in the previous proof.
\end{proof}

\begin{lemma}\label{lem:integralSemi}
    $\displaystyle{
        \int_0^1\int_{B_1(0)} \absOd{\FG G(f)}^2(x,s) \dx \ds \leq \tnorm{f}{Y_1}^2}$.
\end{lemma}

\begin{proof}
    Thanks to Proposition \ref{prop:MildWeak}, we know that $w := G(f) \in L^\infty((0,T) \times \R^n)$ solves
    \begin{align*}
        \partial_t w + \FLhalf w = f \quad \text{in } \mathcal{D}^\prime((0,T) \times \R^n).
    \end{align*}
    We extend $w$ by zero to a function $w : \R^n \times \R \to \R^m$. Using a smoothing kernel $\psi \in C^\infty_c(\R^{n+1})$, we define
    via space-time convolution
    \begin{align*}
        w_\varepsilon := w * \psi_\varepsilon \quad \text{and} \quad 
        f_\varepsilon := f * \psi_\varepsilon
    \end{align*}
    that smoothly solve
    \begin{align}\label{eq:approxPDE}
        \partial_t w_\varepsilon + \FLhalf w_\varepsilon = f_\varepsilon, \quad 
        \text{in } (C_\varepsilon,T-C_\varepsilon) \times \R^n 
    \end{align}
    where $C_\varepsilon>0$ depends on $\varepsilon$ and converges to zero as $\varepsilon \to 0$.
    Multiplying by $\eta^2 w_\varepsilon$, where $\eta \in C^\infty_c(B_2(0))$ with $\eta = 1$ on $B_1(0)$, we obtain the identity
    \begin{align*}
        \int_{C_\varepsilon}^{1-C_\varepsilon} \int_{\R^n} 
        \eta^2 w_\varepsilon \partial_t w_\varepsilon \dx \dt 
        + \int_{C_\varepsilon}^{1-C_\varepsilon} \int_{\R^n} 
        \eta^2 w_\varepsilon \FLhalf w_\varepsilon \dx \dt \\
        = \int_{C_\varepsilon}^{1-C_\varepsilon} \int_{\R^n} 
        \eta^2 w_\varepsilon f_\varepsilon \dx \dt.
    \end{align*}
    Using integration by parts, the identity $\FG \cdot \FG = \FLhalf$ and the fundamental theorem of calculus it becomes
    \begin{align*}
        \frac{1}{2} \int_{\R^n} \eta^2 \abs{w_\varepsilon}^2 
        \Big|_{t = C_\varepsilon}^{t = 1-C_\varepsilon} \dx 
        + \int_{C_\varepsilon}^{1-C_\varepsilon} \int_{\R^n} 
        \eta^2 \absOd{\FG w_\varepsilon}^2 \dx \dt \\
        = \int_{C_\varepsilon}^{1-C_\varepsilon} \int_{\R^n} 
        \eta^2 w_\varepsilon f_\varepsilon \dx \dt 
        + \int_{C_\varepsilon}^{1-C_\varepsilon} \int_{\R^n} 
        w_\varepsilon \scp{\FG (\eta^2)}{\FG w_\varepsilon} \dx \dt.
    \end{align*}
    This can be further arranged and estimated as
    \begin{align*}
        \int_{C_\varepsilon}^{1-C_\varepsilon} \int_{B_1(0)} 
        \absOd{\FG w_\varepsilon}^2 \dx \dt \leq 
        C \sup_{y \in \R^n} \int_{C_\varepsilon}^{1-C_\varepsilon} 
        \fint_{B_1(y)} \abs{f_\varepsilon} \dx \dt \sup_{t}\norm{w_\varepsilon}{L^\infty} \\
        + C \sup_{t}\norm{w_\varepsilon}{L^\infty}^2 + C \sup_{t}\norm{w_\varepsilon}{L^\infty} t^{\frac{1}{2}}\norm{\FG w_\varepsilon}{L^\infty_{od}}.
    \end{align*}
    Taking the limit inferior, using Fatou’s lemma, applying the Hausdorff-Young inequality and Lemma \ref{lem:HausdorfYoungUloc} we get
    the estimate
    \begin{align*}
        \int_0^1 \int_{B_1(0)} \absOd{\FG w}^2 \dx \dt \leq 
        c \sup_{y \in \R^n} \int_{C}^{1} \fint_{B_1(y)} \abs{f} \dx \dt \sup_{t}\norm{w}{L^\infty} \\
        + c \sup_{t}\norm{w}{L^\infty}^2 + c\sup_{t}\norm{w_\varepsilon}{L^\infty} t^{\frac{1}{2}}\norm{\FG w_\varepsilon}{L^\infty_{od}}\leq C\tnorm{f}{Y_1}^2.
    \end{align*}
\end{proof}

\begin{lemma}
    \begin{align}\label{eq:gradientY}
        \int_0^1\int_{B_1(0)} s \abs{\nabla G(f)}^2(x,s) \dx \ds \leq \tnorm{f}{Y_1}^2
    \end{align}
    \begin{align}\label{eq:GradFracGradY}
        \int_0^1\int_{B_1(0)} s^2\absOd{\FG \nabla G(f)}^2(x,s) \dx \ds \leq \tnorm{f}{Y_1}^2
    \end{align}
\end{lemma}

\begin{proof}
    The argument is similar to the one in Lemma~\ref{lem:integralSemi}. Estimate \eqref{eq:gradientY} requires to analyze
    \begin{align*}
        \partial_t \FG w_\varepsilon + \FG \FLhalf w_\varepsilon = \FG f_\varepsilon,
    \end{align*}
    i.e., testing in the \( L^2_{od} \) sense with \( t \eta^2 \FG w_\varepsilon \). For the resulting estimates, we use the interpolation inequality \eqref{eq:nonsnorminterpolationLaplace}.
    For \eqref{eq:GradFracGradY}, we utilize \eqref{eq:gradientY} and analyze
    \begin{align*}
        \partial_t \nabla w_\varepsilon + \FLhalf \nabla w_\varepsilon = \nabla f_\varepsilon.
    \end{align*}
    The result follows by applying the same argumentation as in Lemma~\ref{lem:integralSemi}.
\end{proof}

\section{Proofs of the main results}

We consider the mild formulation of \eqref{eq:HHMHF} 
\begin{align*}
    u(t) = \semiGroup{t} a + \mathcal{N}(u)(t),
\end{align*}
where $\mathcal{N}(u) = G\left(u \absOd{\FG u}^2\right)$, and its modification
\begin{align*}
    \mathcal{N}_{\varphi}(u):= G\left(\varphi(u) \absOd{\FG u}^2\right)
\end{align*}
where $\varphi \in C^\infty_c(B_2(0); B_{\frac{3}{2}}(0))$ is such that $\varphi(y) = y$ for $y \in B_{\frac{3}{2}}(0)$. Proposition~\ref{prop:XTYTEstimate} implies the following bound:

\begin{corollary}\label{coro:quadratEstimate}
    For $u \in X_T$ we have $\mathcal{N}_{\varphi}(u) \in X_T$ and \begin{align*}
        \tnorm{\mathcal{N}_{\varphi}(u)}{X_T} \leq c_4\snorm{u}{X_T}^2
    \end{align*}
    for a constant $c_4>0$ that only depends on the dimension.
\end{corollary}

\begin{proposition}\label{prop:localLipschitz}
    Let $u, v \in X_T$. Then we have the estimate
    \begin{align*}
        \tnorm{\mathcal{N}_{\varphi}(u) - \mathcal{N}_{\varphi}(v)}{X_T} 
        \leq C_L(\snorm{u}{X_T}, \snorm{v}{X_T}, n) \tnorm{u-v}{X_T},
    \end{align*}
    where 
    \begin{align*}
        C_L(\snorm{u}{X_T}, \snorm{v}{X_T}) 
        &= K(n) \big(\snorm{u}{X_T} + \snorm{v}{X_T} + \snorm{u}{X_T}\snorm{v}{X_T} \\
        &\quad + \snorm{v}{X_T}^2 + \snorm{u}{X_T}^2 + \snorm{u}{X_T}^2\snorm{v}{X_T}\big).
    \end{align*}
for a constant depending only on $\varphi$ and the dimension.
\end{proposition}

\begin{proof}
    The argument follows from a series of similar estimates similar to \cite{Wang2010}. To illustrate the approach, we provide a detailed derivation of one representative case:  
    \begin{align*}
        &\esssup_{\substack{0 < t < T \\ x \in \R^n}} 
        t^2 \abs{\scp{\FG u}{\FG \nabla u}_{od} (\varphi(u) - \varphi(v))} \\
        \leq &\esssup_{\substack{0 < t < T \\ x \in \R^n}} 
        t^2 \abs{\scp{\FG u}{\FG \nabla u}_{od}} 
        \norm{\nabla \varphi}{L^\infty} \abs{u - v}. \\
        \leq &\esssup_{\substack{0 < t < T \\ x \in \R^n}} 
        t^2 \absOd{\FG u} \absOd{\FG \nabla u} 
        \norm{\nabla \varphi}{L^\infty} \esssup_{\substack{0 < t < T \\ x \in \R^n}} \abs{u - v}. \\
        \leq &\norm{\nabla \varphi}{L^\infty} 
        \snorm{u}{X_T^{(0)}} \snorm{u}{X_{T}^{(1)}} 
        \esssup_{\substack{0 < t < T \\ x \in \R^n}} \abs{u - v}. \\
        \leq &c \norm{\nabla \varphi}{L^\infty} 
        \snorm{u}{X_T}^2 \tnorm{u - v}{X_T}.
    \end{align*}
    In general, some estimates require the use of \eqref{eq:snorminter} due to the appearance of $\nabla u$ on its own. \\
    The same type of estimate applies to all remaining cases, which completes the proof.
\end{proof}

\subsection{Existence for the modified equation}
We first consider 
    \begin{align} \label{eq:modified}
        \begin{cases}
            \partial_t u + \FLhalf u = \varphi(u) \absOd{\FG u}^2 \\
            u(0) = a.
        \end{cases}
    \end{align}
We select $0 < \varepsilon < \frac{1}{c_4}$ with the constant $c_4$ from Corollary \ref{coro:quadratEstimate} such that for the constant $C_L$ from Proposition \ref{prop:localLipschitz} we have 
    \begin{align*}
        C_L(r,r) < \frac{1}{2} \quad \text{for } r < 2\varepsilon.
    \end{align*}
Assuming $[a]_{\mathcal{A}_T} <\eps$ for some $T>0$, we show by induction that for all $l \in \N_0$
    \begin{align*}
        \snorm{u_l}{X_T} \leq 2 \varepsilon,
    \end{align*}
    where the sequence $u_l$ is recursively defined by $u_0=a$ and
    \begin{align*}
        u_l(t) := \semiGroup{t}a + \mathcal{N}_{\varphi}(u_{l-1}).
    \end{align*}
    But through our assumptions on $\varepsilon$ and $a$ we obtain
    \begin{align*}
        \snorm{u_{l+1}}{X_T} \leq \snorm{\semiGroup{s}a}{X_T} + \snorm{\mathcal{N}_{\varphi}(u_l)}{X_T} \leq \varepsilon  + c_4\snorm{u_l}{X_T}^2 \leq 2 \varepsilon
    \end{align*}
   which implies the claim. Consequently, we can apply Proposition \ref{prop:localLipschitz} to infer
    \begin{align*}
        \tnorm{\mathcal{N}_{\varphi}(u_{l+1}) - \mathcal{N}_{\varphi}(u_l)}{X_T} \leq C(\snorm{u_{l+1}}{X_T},\snorm{u_l}{X_T}) \tnorm{u_{l+1} - u_l}{X_T} \leq \frac{1}{2} \tnorm{u_{l+1} - u_l}{X_T}.
    \end{align*}
    Since $X_T$ is a Banach space, standard arguments imply that $u_l$ converges to a solution $u \in X_T$.

\subsection{Spherical values}
To conclude the proof of Theorem~\ref{thm:existenceSph} it remains to show that a solution $u \in X_T$ to \eqref{eq:modified}
with $u(0)=a \in \mathcal{A}$ such that $|a|=1$ satisfies
    \begin{align*}
        \abs{u} = 1 \quad \text{in } [0,T] \times \R^n.
    \end{align*}
Since $\varphi(y) = y$ for $y \in B_{\frac{3}{2}}(0)$ and 
    \begin{align*}
        \norm{u(t)}{L^\infty} \leq \norm{\semiGroup{t-t_0}u(t_0)}{L^\infty} + \norm{\mathcal{N}_{\varphi}(u)(t)}{L^\infty} \leq 1 + c_4\snorm{u}{X_T}^2 \leq \frac{3}{2}
    \end{align*} 
 we have $\varphi(u) = u$ so that $u$ is indeed a solution of the original equation. Taking into account Proposition \ref{prop:MildWeak} and
 the regularity properties of $u \in X_T$ in $\R^n \times (0,T)$, we obtain for $v = \abs{u}^2 - 1$
    \begin{align}\label{eq:calculations}
        \partial_t v &= \partial_t( \abs{u}^2 - 1) = 2u \partial_t u = -2u \FLhalf u + 2 \abs{u}^2 \absOd{\FG u}^2 \\
        &= 2v \absOd{\FG u}^2 - \FLhalf v.
    \end{align}
     Indeed, for the validity of \eqref{eq:calculations} we observe $\FLhalf u \in C^{\frac{1}{2} - \varepsilon}(\R^n)$, for all $\varepsilon > 0$ and thus the fractional Laplacian is represented by a singular integral \cite{Kwasnicki2017}. This immediately follows through Lemma~\ref{lem:BesovInterpolation} and the fact that
    \begin{align}
\FL{s}: B^{t}_{p,q}(\R^n) \to B^{t-2s}_{p,q}(\R^n),
\end{align}
see e.g. Proposition 3.2~\cite{LemariePierre2002}. 
    Thus $v$ solves the initial value problem
    \begin{align*}
        \begin{cases}
            \partial_t v + \FLhalf v = 2v \absOd{\FG u}^2 \\
            v(t_0) = 0
        \end{cases}
    \end{align*}
    in the mild sense taking into account Proposition \ref{prop:MildWeak}.
    It remains to show uniqueness in the class $L^\infty((0,T) \times \R^n)$. Given two such solutions $v$ and $w$ we have
    \begin{align*}
        \sup_{t \in[0,T]}\norm{v(t)-w(t)}{L^\infty} &= \norm{G\left((v - w)\absOd{\FG u}^2\right)}{X_T} \\ &\leq \norm{(v-w)\absOd{\FG u}^2}{Y_T} \leq \sup_{t \in[0,T]}\norm{v(t)-w(t)}{L^\infty} c_4 \snorm{u}{X_T}^2 \\&\leq \frac{1}{2} \sup_{t \in[0,T]}\norm{v(t)-w(t)}{L^\infty}.
    \end{align*}
    Thus, uniqueness follows, and hence $v = 0$.

\subsection{Uniqueness}

\begin{proof}[Proof of Theorem \ref{thm:uniqueness}]
    We show that for two mild solution $u, v$ with $u(0) = v(0)$,
    \begin{align*}
        I := \{t \in [0,T] \, \vert \, u = v \, \text{ on } \, [0,t] \times \R^n\}
    \end{align*}
    is open and closed in $[0,T]$. Closeness follows from the continuity of $u$ and $v$. For the openness we use
    \begin{align*}
        \tnorm{u-v}{X_{T}} \leq C_L(\snorm{u}{X_T},\snorm{v}{X_T}) \tnorm{u-v}{X_T} \leq \frac{1}{2} \tnorm{u-v}{X_T}.
    \end{align*}
    Where we have chosen $\delta = \varepsilon$. Which directly proves the openness of $I$. Moreover, since $u(0) = v(0)$, $I$ is non-empty; hence $I = [0,T]$.
\end{proof}

\subsection{Continuous dependence on initial data}

\begin{proof}[Proof of Theorem \ref{thm:ContDep}]
Let $(a_n)_{n \in \N} \subset \mathcal{V}$ be a sequence and $a \in \mathcal{V}$ such that
\[
    \norm{a_n - a}{L^\infty} + \snorm{a_n - a}{\mathcal{A}_\infty} \to 0 \quad \text{as } n \to \infty.
\]
Let $u_n = v_n + w_n$ and $u = v + w$ be the respective solutions corresponding to $a_n$ and $a$ with $v_n=S(\cdot)a_n$ 
and $v=S(\cdot)a$, which we estimate separately:
\[
    \tnorm{u_n - u}{X_T^{(0)}} \leq \tnorm{v_n - v}{X_T^{(0)}} + \tnorm{w_n - w}{X_T^{(0)}},
\]
where $\tnorm{\cdot}{X_T^{(0)}}= \|\cdot\|_{L^\infty} + \snorm{\cdot}{X_T^{(0)}}$. By Lemma~\ref{lem:initialCond}, we obtain the estimate
\[
    \tnorm{v_n - v}{X_T^{(0)}} \leq c \left( \snorm{a_n - a}{\mathcal{A}_\infty} + \norm{a_n - a}{L^\infty} \right).
\]
For the nonlinear part, we have the Lipschitz estimate
\[
    \tnorm{w_n - w}{X_T^{(0)}} \leq C_{L}^{(0)}(\snorm{u_n}{X_T^{(0)}}, \snorm{u}{X_T^{(0)}}) \tnorm{u_n - u}{X_T^{(0)}},
\]
where $C_{L}^{(0)}$ is the Lipschitz constant considering only the $\snorm{\cdot}{X_T^{(0)}}$ part, similar to Proposition \ref{prop:localLipschitz}. If $\snorm{u_n}{X_T^{(0)}}, \snorm{u}{X_T^{(0)}}$ are small enough, we can ensure $C_{L}^{(0)} < \tfrac{1}{2}$, and thus
\[
    \tnorm{u_n - u}{X_T^{(0)}} \leq 2c \left( \snorm{a_n - a}{\mathcal{A}_\infty} + \norm{a_n - a}{L^\infty} \right) \to 0 \quad \text{as } n \to \infty.
\]
We now choose $N \in \N$ large enough such that
\[
    \snorm{a_n}{\mathcal{A}_\infty}, \snorm{a}{\mathcal{A}_\infty} < \varepsilon \quad \text{for all } n \geq N,
\]
which is possible since $a_n \to a$ in $L^\infty\cap\mathcal{A}$. Therefore, there exists a common $T > 0$ such that all solutions $u_n, u$ exist on $[0,T]$.
\medskip
Now, by Corollary~\ref{coro:quadratEstimate}, we obtain the quadratic estimate:
\[
    0 \leq c_4 \snorm{u_n}{X_T^{(0)}}^2 - \snorm{u_n}{X_T^{(0)}} + \snorm{v_n}{X_{T_1}^{(0)}},
\]
valid for any $T_1 \geq T$. Since the map $T \mapsto \snorm{u_n}{X_T^{(0)}}$ is continuous (as $u_n \in X_{0,T}$), we obtain, for $T \leq T_1$, the bound
\[
    0 \leq \snorm{u_n}{X_T^{(0)}} \leq \frac{1}{2c_4} - \sqrt{ \frac{1}{4c_4^2} - \frac{1}{c_4} \snorm{v_n}{X_{T_1}^{(0)}}}.
\]
Hence, the semi-norms $\snorm{u_n}{X_T^{(0)}}$ are uniformly bounded in terms of $\snorm{v_n}{X_T^{(0)}}, \snorm{v}{X_T^{(0)}}$. Therefore, for $N$ sufficiently large and $T$ sufficiently small, we obtain a suitably small $\tilde{\varepsilon} > 0$ such that
\[
    \snorm{v_n}{X_T^{(0)}}, \snorm{v}{X_T^{(0)}} < \tilde{\varepsilon} \quad \text{which implies} \quad \snorm{u_n}{X_T^{(0)}}, \snorm{u}{X_T^{(0)}} < \delta,
\]
for some $\delta > 0$ chosen such that $C_{L}^{(0)} < \tfrac{1}{2}$. This concludes the proof.
\end{proof}


\subsection{Proof of Theorem \ref{thm:Q0Included}}

Due to the translation and scaling-invariance of \( Q_0(\R^n) \), it is sufficient to show that  
    \begin{align*}
        \int_0^1 \fint_{B_1(0)} \absOd{\FG (p_t * a)}^2 \dx \dt \leq c \snorm{a}{Q_0}^2.
    \end{align*}
    We split this integral into two parts:
    \begin{align*}
        \int_0^1 \fint_{B_1(0)} \absOd{\FG (p_t * a)}^2 \dx \dt &\leq 
        2\int_0^1 \fint_{B_1(0)} \left(\absOd{\FG (p_t * a)}^{(0,t)}\right)^2 \dx \dt \\&+ 2\int_0^1 \fint_{B_1(0)} \left(\absOd{\FG (p_t * a)}^{(t,\infty)}\right)^2 \dx \dt.
    \end{align*}
    The first part is estimated in a similar way to how one estimates Carleson measures using the BMO norm. The details of the proof are comparable to Theorem 3.3.8 in \cite{GrafakosLoukas2014}. \\
    The key difference is that, for the Fourier argument, we estimate
    \begin{align*}
        \left(\absOd{\FG (p_t * \chi_Q(a - \overline{a}_Q))}^{(0,t)}\right)^2 \leq \absOd{\FG (p_t * \chi_Q(a - \overline{a}_Q)}^2,
    \end{align*}
    and for the growth part, we pull the derivative inside and use the fact that
    \begin{align*}
        \absOd{\FG p_t}^{(0,t)} = \frac{1}{\sqrt{t}} \left(\absOd{\FG p_1}^{(0,1)}\right)_t,
    \end{align*}
    where through Lemma~\ref{lem:growthEstimate} \( \absOd{\FG p_1}^{(0,1)} \) decays as \( \frac{1}{1+\abs{x}^{n+2}} \), which is suited for the second part of Theorem 3.3.8 in \cite{GrafakosLoukas2014}. \\
    That is, we have
    \begin{align*}
        \int_0^1 \fint_{B_1(0)} \left(\absOd{\FG (p_t * a)}^{(0,t)}\right)^2 \dx \dt \leq c\snorm{a}{\mathrm{BMO}}^2 \leq c\snorm{a}{Q_0}^2.
    \end{align*}
    For the second integral, we immediately pull the derivative into the convolution, which leads to  
    \begin{align*}
        \int_0^1 \fint_{B_1(0)} \left(p_t * \absOd{\FG a}^{(t,\infty)}\right)^2 \dx \dt.
    \end{align*}
    We cannot directly apply Lemma \ref{lem:HausdorfYoungUloc} here because both \( p_t \) and  
    \( u := \absOd{\FG a}^{(t,\infty)} \) depend on \( t \). Nevertheless, we obtain a similar estimate as follows. \\
    We decompose \( u = u_1 + u_2 \), where  
    \begin{align*}
        u_1 = u \chi_{\abs{x} \leq 2}, \quad u_2 = u \chi_{\abs{x} > 2}.
    \end{align*}
    The estimate for \( u_1 \) follows from applying the standard Hausdorff-Young convolution inequality, leading to:
    \begin{align*}
        \int_0^1 \fint_{B_1(0)} \left(p_t * u_1\right)^2 \dx \dt \leq \sup_{t \in [0,1]} \norm{p_t}{L^1}^2 \int_0^1 \int_{\R^n} \abs{u_2}^2 \dx \dt.
    \end{align*}
    For \( u_2 \), we proceed as follows:
    \begin{align*}
        \fint_{B_1(0)} (p_t * u_2)^2 \dx &=  \fint_{B_1(0)}
        \left(\int_{\abs{y} \geq 2} \frac{t \abs{u(y)}}{(t^2 + \abs{x-y}^2)^{\frac{n+1}{2}}}\dy\right)^2 \dx \\
        &\leq c \fint_{B_1(0)}\left(\int_{\abs{y} \geq 2} \frac{\abs{u(y)}}{\abs{y}^{n+1}} \dy\right)^2 \dx.
    \end{align*}
    In the last estimate, we used the fact that for $\abs{x} \leq 1 \leq \frac{\abs{y}}{2}$
    \begin{align*}
        \abs{x-y} \geq \frac{\abs{y}}{2}.
    \end{align*}
    Now, applying the remaining integral and using Jensen's inequality with respect to the measure \(\chi_{\abs{y} \geq 2}\frac{\dy}{\abs{y}^{n+1}}\), we obtain  
    \begin{align*}
        \int_0^1 \int_{B_1(0)} \abs{\int_{\abs{y} \geq 2} 
        \frac{\abs{u(y)}}{\abs{y}^{n+1}} \dy}^2 \dx \dt
        &\leq c \int_0^1 \int_{\abs{y} \geq 2} 
        \frac{\abs{u(y)}^2}{\abs{y}^{n+1}} \dy \dt \\
        &\leq c \sup_{x \in \R^n} \int_0^1 \int_{B_1(x)} 
        \abs{u(y)}^2 \dy \dt.
    \end{align*}
    Thus, we have  
    \begin{align*}
        \int_0^1 \fint_{B_1(0)} \left(p_t * \absOd{\FG a}^{(t,\infty)}\right)^2 \dx \dt \leq  c \sup_{x \in \R^n} \int_0^1 \int_{B_1(x)} 
        \left(\absOd{\FG a}^{(t,\infty)}\right)^2 \dy \dt.
    \end{align*}
    By translation invariance, it suffices to show that  
    \begin{align*}
        \int_0^1 \int_{B_1(0)} 
        \left(\absOd{\FG a}^{(t,\infty)}\right)^2 \dx \dt 
        \leq \snorm{a}{Q_0}^2.
    \end{align*}
    We immediately observe that  
    \begin{align*}
        \int_0^1 \int_{B_1(0)} 
        \left(\absOd{\FG a}^{(t,\infty)}\right)^2 \dx \dt 
        &= \int_0^1 \int_{B_1(0)} \int_{\R^n} 
        \chi_{\abs{h} \geq t} \frac{\abs{a(x+h) - a(x)}^2}{\abs{h}^{n+1}} 
        \dh \dx \dt.
    \end{align*}
    Due to the non-negativity of the integrand, we apply Fubini’s theorem to obtain
    \begin{gather*}
        \int_0^1 \int_{B_1(0)} \int_{\R^n} 
        \chi_{\abs{h} \geq t} \frac{\abs{a(x+h) - a(x)}^2}{\abs{h}^{n+1}} 
        \dh \dx \dt \\
        = \int_{B_1(0)} \int_{\R^n} 
        \int_0^1 \chi_{\abs{h} \geq t} \dt \frac{\abs{a(x+h) - a(x)}^2}{\abs{h}^{n+1}} 
        \dh \dx \\
        = \int_{B_1(0)} \int_{\R^n} 
        \min(\abs{h},1) \frac{\abs{a(x+h) - a(x)}^2}{\abs{h}^{n+1}} 
        \dh \dx.
    \end{gather*}
    This integral decomposes into two parts:  
    \begin{align*}
        \int_{B_1(0)} \int_{\R^n \setminus B_1(0)} 
        \frac{\abs{a(x+h) - a(x)}^2}{\abs{h}^{n+1}} \dh \dx 
        + \int_{B_1(0)} \int_{B_1(0)} \frac{\abs{a(x+h) - a(x)}^2}{\abs{h}^n} 
        \dh \dx.
    \end{align*}
    The first integral we immediately obtain
    \begin{align*}
        \int_{B_1(0)} \int_{B_1(0)} \frac{\abs{a(x+h) - a(x)}^2}{\abs{h}^n} \dh \dx 
        &\leq c\snorm{a}{Q_0}^2.
    \end{align*}
    For the second integral, we note that by the embedding \(\mathrm{BMO} \supset Q_0 \), we have
    \begin{align*}
        \int_{B_1(0)} \int_{\R^n \setminus B_1(0)} 
        \frac{\abs{a(x+h) - a(x)}^2}{\abs{h}^{n+1}} \dh \dx 
        &\leq c\snorm{a}{BMO}^2 \leq c\snorm{a}{Q_0}^2.
    \end{align*}
    The second estimate follows from applying the triangle inequality to the averages  $\overline{a}_{B(h)}$ and $\overline{a}_{B(0)}$,  together with the standard bound (see Exercise 3.1.6 in \cite{GrafakosLoukas2014}):  
    \begin{align*}
        \abs{\overline{a}_{B(h)} -\overline{a}_{B(0)}}^2 
        \leq C_n \log(\abs{h})^2 \snorm{a}{\mathrm{BMO}}^2, 
        \quad \text{for large } h.
    \end{align*}
    Since we have established both estimates, we conclude  
    \begin{align*}
        \int_0^1 \int_{B_1(0)} 
        \left(\absOd{\FG a}^{(t,\infty)}\right)^2 \dx \dt 
        \leq c \snorm{a}{Q_0}^2.
    \end{align*}
    The proof is complete.


\section*{Acknowledgements}
This work is funded by the Deutsche Forschungsgemeinschaft (DFG, German Research Foundation) – project number \href{https://gepris.dfg.de/gepris/projekt/442047500}{442047500} – through the Collaborative Research Center \href{https://sfb1481.rwth-aachen.de/}{"Sparsity and Singular Structures" (SFB 1481)}. KK further thanks \href{https://www.math1.rwth-aachen.de/cms/MATH1/Der-Lehrstuhl/Team/Wissenschaftlich-Beschaeftigte/~bjomyl/Valentin-Linse/}{Valentin Linse} for the fruitful discussions.

\bibliographystyle{plainurl}
\bibliography{bib.bib} 

\end{document}